\documentclass[11pt, final]{article}
\usepackage{a4}
\usepackage{amsmath}%
\usepackage{amstext}%
\usepackage{amssymb}%
\usepackage{showkeys}%
\usepackage{epsfig}%
\usepackage{cite}
\usepackage{tikz}
\usepackage{subfig}
\usepackage{caption}
\usepackage{multirow}
\usepackage{booktabs}
\usepackage{floatrow}

\usepackage{listings}
\usepackage[margin=1.2in]{geometry}
\newtheorem{theorem}{Theorem}

\newtheorem{algorithm}[theorem]{Algorithm}

\newtheorem{lemma}[theorem]{Lemma}

\newtheorem{pb}[theorem]{Problem}
\newenvironment{problem}{\pb\rm}{\endpb}

\newtheorem{rem}[theorem]{Remark}
\newenvironment{remark}{\rem\rm}{\endrem}

\newcounter{unnumber}

\newtheorem{prf}[unnumber]{Proof.}
\newenvironment{proof}{\prf\rm}{\hfill{$\blacksquare$}\endprf}
\newcommand{\R}{\mathbb{R}}%
\newcommand{\N}{\mathbb{N}}%
\newcommand{\ol}{\overline}%

\DeclareMathOperator*\inte{int}%
\DeclareMathOperator*\sqri{sqri}%
\DeclareMathOperator*\ri{ri}%
\DeclareMathOperator*\dom{dom}%
\DeclareMathOperator*\B{\overline{\R}}%
\DeclareMathOperator*\gr{Gr}%
\DeclareMathOperator*\ran{ran}%
\DeclareMathOperator*\id{Id}%
\DeclareMathOperator*\prox{prox}%
\DeclareMathOperator*\argmin{argmin}

\title{ADMM for monotone operators: convergence analysis and rates}

\author{Radu Ioan Bo\c{t} \thanks{University of Vienna, Faculty of Mathematics, Oskar-Morgenstern-Platz 1, A-1090 Vienna, Austria,
email: radu.bot@univie.ac.at. Research partially supported by FWF (Austrian Science Fund), project I 2419-N32.} \and
Ern\"{o} Robert Csetnek \thanks {University of Vienna, Faculty of Mathematics, Oskar-Morgenstern-Platz 1, A-1090 Vienna, Austria,
email: ernoe.robert.csetnek@univie.ac.at. Research supported by FWF (Austrian Science Fund), project P 29809-N32.}}

\begin{document}
\maketitle

\noindent \textbf{Abstract.} We propose in this paper a unifying scheme for several algorithms from the literature dedicated to the 
solving of monotone inclusion problems involving compositions with linear continuous operators in infinite dimensional 
Hilbert spaces. We show that a number of primal-dual algorithms for monotone inclusions and also  
the classical ADMM numerical scheme for convex optimization problems, along with some of its variants, 
can be embedded in this unifying scheme. While in the first part of the paper convergence results for the iterates are reported, 
the second part is devoted to the derivation of convergence rates obtained by combining variable metric techniques with strategies 
based on suitable choice of dynamical step sizes. \vspace{1ex}

 \vspace{1ex}


\section{Introduction and preliminaries}\label{sec1}

Consider the convex optimization problem 
\begin{equation}\label{prim-h-intr} \inf_{x\in{\cal H}}\{f(x)+g(Lx) + h(x)\},
\end{equation}
where ${\cal H}$ and ${\cal G}$ are real Hilbert spaces, $f:{\cal H}\to \overline \R :=\R\cup\{\pm\infty\}$ and  $g:{\cal G}\to \overline \R$ are proper, convex and lower semicontinuous functions, 
$h:{\cal H}\rightarrow\R$ is a convex and Fr\'{e}chet differentiable function with Lipschitz continuous gradient and 
$L:{\cal H}\rightarrow{\cal G}$ is a linear continuous operator.  

Due to  numerous applications in fields like signal and image processing, portfolio optimization, cluster analysis, location theory, network communication, machine learning, the design and investigation of 
numerical algorithms for solving convex optimization problems of type \eqref{prim-h-intr} 
attracted in the last couple of years huge interest from the applied mathematics community.  The most prominent methods one can find in the literature for solving \eqref{prim-h-intr}
are the {\it primal-dual proximal splitting algorithms} and the {\it ADMM algorithms}. 
We briefly describe the two classes of algorithms.

First proximal splitting algorithms for solving convex optimization problems involving compositions 
with linear continuous operators  have been reported by Combettes and Ways \cite{comb-w}, Esser, Zhang and Chan \cite{esser} and Chambolle and Pock \cite{ch-pck}. Further investigations have been made in 
the more general framework of finding zeros of sums of linearly composed maximally monotone operators, and monotone and Lipschitz, respectively, cocoercive operators. The resulting numerical schemes 
have been employed in the solving of the inclusion problem
\begin{equation}\label{inclusion}
\mbox{find} \ x \in {\cal H} \ \mbox{such that} \ 0 \in \partial f(x) + (L^* \circ \partial g \circ L)(x) + \nabla h(x),
\end{equation}
which represents the system of optimality conditions of problem \eqref{prim-h-intr}.

Brice\~{n}o-Arias and Combettes pioneered this approach in \cite{br-combettes}, by reformulating the general monotone inclusion in an appropriate product space  as the sum 
of a maximally monotone operator and a linear and skew one, and by solving the resulting inclusion problem via a forward-backward-forward type algorithm (see also \cite{combettes-pesquet}). 
Afterwards, by using the same product space approach, this time 
in a suitable renormed space, V\~{u} succeeded in \cite{vu} in formulating a primal-dual splitting algorithm of forward-backward type, in other words, by saving a forward step. 
Condat has presented in \cite{condat2013} in the variational case algorithms of the same nature with the one in \cite{vu}. Under strong monotonicity/convexity assumptions and the use of 
dynamic step size strategies convergence rates have been provided in \cite{b-c-h2} for the primal-dual algorithm in \cite{vu} (see also \cite{ch-pck}), and in \cite{b-h} for the primal-dual 
algorithm in \cite{combettes-pesquet}. 

We describe the ADMM algorithm for solving \eqref{prim-h-intr} in the case $h=0$, which corresponds to the standard setting in the literature. By introducing an auxiliary variable one can rewrite \eqref{prim-h-intr} as  
\begin{equation}\label{prim-x-z}  \inf_{\substack{(x,z)\in\cal{H}\times\cal{G}\\Lx-z=0}}\{f(x)+g(z)\}.
\end{equation} 
For a fixed real number $c >0$ we consider the augmented Lagrangian associated with problem \eqref{prim-x-z}, 
which is defined as
$$L_c:{\cal{H}} \times {\cal{G}} \times {\cal{G}} \rightarrow \B, \ L_c(x,z,y)=f(x)+g(z)+\langle y,Ax-z\rangle + \frac{c}{2} \|Ax-z \|^2.$$ 
The ADMM algorithm relies on the alternating minimization of the augmented Lagrangian with respect to the variables $x$  and $z$ (see \cite{bpcpe, gabaymercier, fortinglowinski, gabay, ecb} and 
Remark \ref{bbc-sh-teb} for the exact formulation of the iterative scheme).  Generally, the minimization with respect to the variable $x$ does not lead to a proximal step. This drawback has been overcome 
by Shefi and Teboulle in \cite{shefi-teboulle2014} by introducing additional suitably chosen metrics, and also in
\cite{bbc-admm} for an extension of the ADMM algorithm designed for problems which involve also smooth parts in the objective.

The aim of this paper is to provide a unifying algorithmic scheme for solving monotone inclusion problems which encompasses several primal-dual iterative methods \cite{vu, condat2013, b-c-h, ch-pck}, 
and the ADMM algorithm (and its variants from \cite{shefi-teboulle2014}) in the particular case of convex optimization problems. A closer look at the structure of the new algorithmic scheme shows that it 
translates the paradigm behind ADMM methods for optimization problems to the solving of monotone inclusions.  We carry out a convergence analysis for the proposed iterative scheme by making use of techniques 
relying on the Opial Lemma applied in a variable metric setting. Furthermore, we derive convergence rates for the iterates under supplementary strong monotonicity assumptions. To this aim we use a dynamic 
step strategy, based on which we can provide a unifying scheme for the algorithms in \cite{b-c-h2, ch-pck}. Not least we also provide accelerated versions for the classical ADMM algorithm (and its variable metric variants). 

In what follows we recall some elements of the theory of monotone operators in Hilbert spaces and refer  for more details to \cite{bauschke-book, bo-van, simons}.

Let ${\cal H}$ be a real Hilbert space with inner product $\langle\cdot,\cdot\rangle$ and associated norm 
$\|\cdot\|=\sqrt{\langle \cdot,\cdot\rangle}$.
For an arbitrary set-valued operator $A:{\cal H}\rightrightarrows {\cal H}$ we denote by 
$\gr A=\{(x,u)\in {\cal H}\times {\cal H}:u\in Ax\}$ its graph, 
by $\dom A=\{x \in {\cal H} : Ax \neq \emptyset\}$ its domain and by $A^{-1}:{\cal H}\rightrightarrows {\cal H}$
its inverse operator, defined by $(u,x)\in\gr A^{-1}$ if and only if $(x,u)\in\gr A$. We say that $A$ is monotone if 
$\langle x-y,u-v\rangle\geq 0$ for all $(x,u),(y,v)\in\gr A$.
A monotone operator $A$ is said to be maximal monotone, if there exists no proper monotone extension of the graph of $A$ on 
${\cal H}\times {\cal H}$. 

The resolvent of $A$, $J_A:{\cal H} \rightrightarrows {\cal H}$, is defined by $J_A=(\id+A)^{-1}$, where 
$\id :{\cal H} \rightarrow {\cal H}, \id(x) = x$ for all $x \in {\cal H}$, is the identity operator on ${\cal H}$. 
If $A$ is maximal monotone, then $J_A:{\cal H} \rightarrow {\cal H}$ is single-valued and maximal monotone (see \cite[Proposition 23.7 and Corollary 23.10]{bauschke-book}). 
For an arbitrary $\gamma>0$ we have (see \cite[Proposition 23.2]{bauschke-book})
$$p\in J_{\gamma A}x \ \mbox{if and only if} \ (p,\gamma^{-1}(x-p))\in\gr A$$
and (see \cite[Proposition 23.18]{bauschke-book})
\begin{equation}\label{j-inv-op}
J_{\gamma A}+\gamma J_{\gamma^{-1}A^{-1}}\circ \gamma^{-1}\id=\id.
\end{equation}

When ${\cal G}$ is another Hilbert space and $L:{\cal H} \rightarrow {\cal G}$ is a linear continuous operator, then 
$L^* : {\cal G} \rightarrow {\cal H}$, defined by $\langle L^*y,x\rangle = \langle y,Lx \rangle$ for all 
$(x,y) \in {\cal H} \times {\cal G}$, denotes the adjoint operator of $L$, 
while the norm of $L$ is defined as $\|L\| = \sup\{\|Lx\|: x \in {\cal H}, \|x\| \leq 1\}$.

Let $\gamma>0$ be arbitrary. We say that $A$ is $\gamma$-strongly monotone, if $\langle x-y,u-v\rangle\geq \gamma\|x-y\|^2$ for 
all $(x,u),(y,v)\in\gr A$. A single-valued operator $A:{\cal H}\rightarrow {\cal H}$ is said to be $\gamma$-cocoercive, if 
$\langle x-y,Ax-Ay\rangle\geq \gamma\|Ax-Ay\|^2$ for all $(x,y)\in {\cal H}\times {\cal H}$. Moreover, $A$ is $\gamma$-Lipschitz continuous,
if $\|Ax-Ay\|\leq \gamma\|x-y\|$ for all $(x,y)\in {\cal H}\times {\cal H}$. A single-valued linear operator 
$A:{\cal H} \rightarrow {\cal H}$ is said to be skew, if $\langle x,Ax\rangle =0$ for all $x \in {\cal H}$. 
The parallel sum of two operators $A,B:{\cal H}\rightrightarrows {\cal H}$ is defined by 
$A\Box B: {\cal H}\rightrightarrows {\cal H}, A\Box B=(A^{-1}+B^{-1})^{-1}$.

Since the variational case will be also in the focus of our investigations, we recall next some elements of convex analysis. 

For a function $f:{\cal H}\rightarrow\overline{\R}$ we denote 
by $\dom f=\{x\in {\cal H}:f(x)<+\infty\}$ its effective domain and say that $f$ is proper, if $\dom f\neq\emptyset$ 
and $f(x)\neq-\infty$ for all $x\in {\cal H}$. We denote by $\Gamma({\cal H})$ the family of proper convex and lower semi-continuous 
extended real-valued functions defined on ${\cal H}$. Let $f^*:{\cal H} \rightarrow \overline \R$, 
$f^*(u)=\sup_{x\in {\cal H}}\{\langle u,x\rangle-f(x)\}$ for all $u\in {\cal H}$, be the conjugate function of $f$. 
The subdifferential of $f$ at $x\in {\cal H}$, with $f(x)\in\R$, is the set 
$\partial f(x):=\{v\in {\cal H}:f(y)\geq f(x)+\langle v,y-x\rangle \ \forall y\in {\cal H}\}$. We take by convention 
$\partial f(x):=\emptyset$, if $f(x)\in\{\pm\infty\}$.   If $f\in\Gamma({\cal H})$, then $\partial f$ is a 
maximally monotone operator (cf. \cite{rock}) and it holds $(\partial f)^{-1} = 
\partial f^*$. For $f,g:{\cal H}\rightarrow \overline{\R}$ two proper functions, we consider also their infimal convolution, 
which is the function $f\Box g:{\cal H}\rightarrow\B$, defined by $(f\Box g)(x)=\inf_{y\in {\cal H}}\{f(y)+g(x-y)\}$, 
for all $x\in {\cal H}$.

When $f\in\Gamma({\cal H})$ and $\gamma > 0$, for every $x \in {\cal H}$ we denote by $\prox_{\gamma f}(x)$ the proximal point
of parameter $\gamma$ of $f$ at $x$, which is the unique optimal solution of the optimization problem
\begin{equation}\label{prox-def}\inf_{y\in {\cal H}}\left \{f(y)+\frac{1}{2\gamma}\|y-x\|^2\right\}.
\end{equation}
Notice that $J_{\gamma\partial f}=(\id +\gamma\partial f)^{-1}=\prox_{\gamma f}$, thus 
$\prox_{\gamma f} :{\cal H} \rightarrow {\cal H}$ is a single-valued operator fulfilling the extended 
Moreau's decomposition formula
\begin{equation}\label{prox-f-star}
\prox\nolimits_{\gamma f}+\gamma\prox\nolimits_{(1/\gamma)f^*}\circ\gamma^{-1}\id =\id.
\end{equation}
Finally, we say that the function $f:{\cal H} \rightarrow \overline \R$ is $\gamma$-strongly convex for $\gamma >0$, 
if $f-\frac{\gamma}{2}\|\cdot\|^2$ is a convex function. This property implies that $\partial f$ is $\gamma$-strongly monotone (see \cite[Example 22.3]{bauschke-book}).

\section{The ADMM paradigm employed to monotone inclusions}\label{sec2}

In this section we propose an algorithm for solving monotone inclusion problems involving compositions with linear continuous operators in infinite dimensional Hilbert spaces which is designed in the 
spirit of the ADMM paradigm.

\subsection{Problem formulation, algorithm and particular cases}\label{sub21}

The following problem represents the central point of our investigations.

\begin{problem}\label{pr-mon}
Let ${\cal H}$ and ${\cal G}$ be real Hilbert spaces, $A:{\cal H}\rightrightarrows {\cal H}$ and 
$B:{\cal G } \rightrightarrows {\cal G}$ be maximally monotone operators and 
$C:{\cal H}\rightarrow {\cal H}$ an $\eta$-cocoercive operator for $\eta>0$. 
Let $L:{\cal H}\rightarrow$ ${\cal G}$ be a linear continuous operator. 
To solve is the primal monotone inclusion
\begin{equation}\label{p-mon}
\mbox{find }  x \in {\cal H} \ \mbox{such that} \ 0\in A x+ (L^*\circ B\circ L) x+Cx,
\end{equation}
together with its dual monotone inclusion
\begin{equation}\label{d-mon}
\mbox{ find } v\in{\cal G} \ \mbox{such that } \exists x\in {\cal H}: -L^*v\in Ax+Cx\mbox{ and } v\in B(Lx).
\end{equation}
\end{problem}

Simple algebraic manipulations yield that \eqref{d-mon} is equivalent to the problem 
\begin{equation*}
\mbox{ find } v\in{\cal G} \ \mbox{such that }  \ 0\in B^{-1}v+ \Big((-L)\circ (A+C)^{-1}\circ (-L^*)\Big)v ,
\end{equation*}
which can be equivalently written as
\begin{equation}\label{d-mon-equiv}
\mbox{ find } v\in{\cal G} \ \mbox{such that }   \ 0\in B^{-1}v+ \Big((-L)\circ (A^{-1}\Box C^{-1})\circ (-L^*)\Big)v. 
\end{equation}

We say that $(x, v)\in{\cal H} \times$ ${\cal G}$ is a primal-dual solution to the primal-dual pair of monotone inclusions \eqref{p-mon}-\eqref{d-mon}, if 
\begin{equation}\label{prim-dual}-L^*v\in Ax+Cx \mbox{ and }v\in B(Lx).\end{equation}

If $x\in{\cal H}$ is a solution to \eqref{p-mon}, then there exists 
$v\in$ ${\cal G}$ such that $(x, v)$ is a primal-dual solution to \eqref{p-mon}-\eqref{d-mon}. On the other hand, if $v\in$ ${\cal G}$ is a solution to \eqref{d-mon}, then there 
exists $x\in{\cal H}$ such that $(x, v)$ is a primal-dual solution to \eqref{p-mon}-\eqref{d-mon}. Furthermore, if 
$(x, v)\in{\cal H} \times$ ${\cal G}$ is a primal-dual solution to \eqref{p-mon}-\eqref{d-mon}, 
then $x$ is a solution to \eqref{p-mon} and $v$ is a solution to \eqref{d-mon}.

Next we relate this general setting to the solving of a primal-dual pair of convex optimization problems.

\begin{problem}\label{pr-var} Let ${\cal H}$ and ${\cal G}$ be real Hilbert spaces, $f\in\Gamma({\cal H})$, 
$g\in\Gamma({\cal G})$, $h:{\cal H}\rightarrow\R$ a convex and 
Fr\'{e}chet differentiable function with $\eta^{-1}$-Lipschitz continuous gradient, for $\eta>0$, and 
$L:{\cal H}\rightarrow{\cal G}$ a linear continuous operator. Consider the primal convex optimization problem 
\begin{equation}\label{prim-h} \inf_{x\in{\cal H}}\{f(x)+h(x)+g(Lx)\}
\end{equation}
and its Fenchel dual problem 
\begin{equation}\label{dual-h}  \sup_{v\in{\cal G}}\{-(f^*\mathop{\Box} h^*)(-L^*v)-g^*(v)\}.\end{equation}
\end{problem}

The system of optimality conditions for the primal-dual pair of optimization problems \eqref{prim-h}-\eqref{dual-h} reads:
\begin{equation}\label{opt-cond} -L^*\ol v-\nabla h(\ol x)\in\partial f(\ol x) \mbox{ and } \ol v\in\partial g(L\ol x), 
\end{equation}
which is actually a particular formulation of \eqref{prim-dual} when
\begin{equation}\label{choice}
A:=\partial f, \ C:=\nabla h, \ B:=\partial g.
\end{equation}
Notice that, due to the Baillon-Haddad Theorem (see 
\cite[Corollary 18.16]{bauschke-book}), $\nabla h$ is $\eta$-cocoercive.

If \eqref{prim-h} has an optimal solution $x\in{\cal H}$ and a suitable qualification condition is fulfilled, 
then there exists $v\in{\cal G}$, an optimal solution to  \eqref{dual-h}, such that  \eqref{opt-cond} holds. If \eqref{dual-h} has an optimal solution $v\in{\cal G}$ and a suitable qualification condition is fulfilled, 
then there exists $x\in{\cal H}$, an optimal solution to  \eqref{prim-h}, such that  \eqref{opt-cond} holds. Furthermore, if the pair 
$(x, v)\in{\cal H}\times{\cal G}$ satisfies relation \eqref{opt-cond}, then $x$ is an optimal solution to \eqref{prim-h}, 
$v$ is an optimal solution to \eqref{dual-h} and the optimal objective values of \eqref{prim-h} and \eqref{dual-h} coincide. 

One of the most popular and useful qualification conditions guaranteeing the existence of a dual optimal solution is the one known under the name Attouch-Br\'{e}zis and which requires that:
\begin{equation}\label{reg-cond} 0\in\sqri(\dom g-L(\dom f))
\end{equation}
holds.
Here, for $S\subseteq {\cal G}$ a convex set,  we denote by
$$\sqri S:=\{x\in S:\cup_{\lambda>0}\lambda(S-x) \ \mbox{is a closed linear subspace of} \ {\cal G}\}$$
its strong quasi-relative interior. The topological interior is contained in the 
strong quasi-relative interior: $\inte S\subseteq\sqri S$, however, in general this inclusion may be strict. 
If ${\cal G}$ is finite-dimensional, then for a nonempty and convex set $S \subseteq   {\cal G}$, one has $\sqri S =\ri S$, which denotes the topological interior of $S$ relative to its affine hull. 
Considering again the infinite dimensional setting, we remark that condition 
\eqref{reg-cond} is fulfilled, if there exists $x'\in\dom f$ such that $Lx'\in \dom g$ and $g$ is continuous at $Lx'$. For further considerations on convex duality  we refer to
\cite{bo-van, b-hab, bauschke-book, EkTem, Zal-carte}.

Throughout the paper the following additional notations and facts will be used. We denote by ${\cal S_+}({\cal H})$ the family of operators $U:{\cal H}\rightarrow {\cal H}$ which are 
linear, continuous, self-adjoint and positive semidefinite. For $U\in {\cal S_+}({\cal H})$ we consider the semi-norm defined by 
$$\|x\|_U^2=\langle x,Ux\rangle \ \forall x\in {\cal H}.$$
The Loewner partial ordering is defined for $U_1,U_2\in {\cal S_+}({\cal H})$ by
$$U_1\succcurlyeq U_2\Leftrightarrow \|x\|_{U_1}^2\geq \|x\|_{U_2}^2  \ \forall x\in {\cal H}.$$
Finally, for $\alpha> 0$, we set  
$${\cal P}_{\alpha}({\cal H}):=\{U\in {\cal S_+}({\cal H}): U\succcurlyeq \alpha\id \}.$$
Let $\alpha>0$, $U\in {\cal P}_{\alpha}({\cal H})$ and $A:{\cal H}\rightrightarrows{\cal H}$ a maximally monotone 
operator. Then the operator $(U+A)^{-1}:{\cal H}\to {\cal H}$ is single-valued with full domain; in other words
$$\mbox{for every} \ x\in {\cal H} \mbox{ there exists a unique }p\in{\cal H}\mbox{ such that }p=(U+A)^{-1}x.$$
Indeed, this is a consequence of the relation $$(U+A)^{-1}=(\id +U^{-1}A)^{-1}\circ U^{-1}$$
and of the maximal monotonicity of the operator $U^{-1}A$ in the renormed Hilbert space $({\cal H},\langle \cdot,\cdot\rangle_U)$ (see for example 
\cite[Lemma 3.7]{combettes-vu2013}), where $$\langle x,y\rangle_U:=\langle x,Uy\rangle \ \forall x,y\in {\cal H}.$$  

We are now in the position to formulate the algorithm relying on the ADMM paradigm for solving the primal-dual pair of monotone inclusions \eqref{p-mon}-\eqref{d-mon}.

\begin{algorithm}\label{alg-mon} For all $k \geq 0$, let $M_1^k\in {\cal S_+}({\cal H})$, $M_2^k\in {\cal S_+}({\cal G})$  and $c>0$ be such 
that $cL^*L+M_1^k\in{\cal P}_{\alpha_k}({\cal H})$ for $\alpha_k>0$.
Choose $(x^0,z^0,y^0)\in{\cal H}\times{\cal G}\times{\cal G}$. For all $k\geq 0$ generate the sequence $(x^k,z^k,y^k)_{k \geq 0}$ as follows:
\begin{eqnarray} x^{k+1} & = &
\left(cL^*L+M_1^k+A\right)^{-1}\left[cL^*(z^k-c^{-1}y^k)+M_1^kx^k-Cx^k\right] \label{C-var-x}\\
z^{k+1} & = &\left(\id+c^{-1}M_2^k+c^{-1}B\right)^{-1}\left[Lx^{k+1}+c^{-1}y^k+c^{-1}M_2^k z^k\right] \label{C-var-z}\\
y^{k+1} & = & y^k+c(Lx^{k+1}-z^{k+1}).\label{C-var-y}
\end{eqnarray}
\end{algorithm}

As shown below, several algorithms from the literature can be embedded in this numerical scheme.  

\begin{remark}\label{bbc-sh-teb} For all $k\geq 0$, the equations \eqref{C-var-x} and \eqref{C-var-z} are equivalent to 
\begin{equation}\label{C-var-x-equiv}cL^*(z^k-Lx^{k+1}-c^{-1}y^k)+M_1^k(x^k-x^{k+1})-Cx^k\in Ax^{k+1},\end{equation}
and, respectively, 
\begin{equation}\label{C-var-z-equiv}c(Lx^{k+1}-z^{k+1}+c^{-1}y^k)+M_2^k(z^k-z^{k+1})\in Bz^{k+1}.\end{equation}
Notice that the latter is equivalent to  
\begin{equation}\label{C-var-z-equiv-equiv-equiv}y^{k+1}+M_2^k(z^k-z^{k+1})\in Bz^{k+1}.\end{equation}

In the variational setting as described in Problem \ref{pr-var}, namely, by choosing the operators as in \eqref{choice}, the inclusion \eqref{C-var-x-equiv} becomes 
$$0\in\partial f(x^{k+1})+cL^*(Lx^{k+1}-z^{k}+c^{-1}y^{k})+M_1^{k}(x^{k+1}-x^{k})+\nabla h(x^{k}),$$
which is equivalent to 
\begin{equation}\label{bbc1}x^{k+1} = 
\argmin_{x\in{\cal H}} \left\{f(x)+\langle x-x^k,\nabla h(x^k)\rangle +\frac{c}{2}\|Lx-z^k+c^{-1}y^k\|^2 + \frac{1}{2}\|x-x^k\|_{M_1^k}^2\right\}.
\end{equation}
On the other hand, \eqref{C-var-z-equiv} becomes 
$$c(Lx^{k+1}-z^{k+1}+c^{-1}y^k)+M_2^k(z^k-z^{k+1})\in\partial g(z^{k+1}),$$
which is equivalent to 
\begin{equation}\label{bbc2}z^{k+1}  = \argmin_{z\in{\cal G}} \left\{g(z)+\frac{c}{2}\|Lx^{k+1}-z+c^{-1}y^k\|^2 + \frac{1}{2}\|z-z^k\|_{M_2^k}^2\right\}. 
\end{equation}
Consequently, the iterative scheme \eqref{C-var-x}-\eqref{C-var-y} reads
\begin{eqnarray} x^{k+1} & = &
\argmin_{x\in{\cal H}} \left\{f(x)+\langle x-x^k,\nabla h(x^k)\rangle +\frac{c}{2}\|Ax-z^k+c^{-1}y^k\|^2 + \frac{1}{2}\|x-x^k\|_{M_1^k}^2\right\} \label{h-var-x}\\
z^{k+1} & = &\argmin_{z\in{\cal G}} \left\{g(z)+\frac{c}{2}\|Ax^{k+1}-z+c^{-1}y^k\|^2 + \frac{1}{2}\|z-z^k\|_{M_2^k}^2\right\} \label{h-var-z}\\
y^{k+1} & = & y^k+c(Ax^{k+1}-z^{k+1}),\label{h-var-y}
\end{eqnarray}
which is the algorithm formulated and investigated by Banert, Bo\c t and Csetnek in \cite{bbc-admm}. 
The case when $h=0$ and $M_1^k, M_2^k$ are constant for every $k\geq 0$ has been considered in the setting of finite dimensional 
Hilbert spaces by Shefi and Teboulle \cite{shefi-teboulle2014}. We want to emphasize that when $h=0$ and $M_1^k=M_2^k=0$ for all 
$k\geq 0$ the iterative scheme \eqref{h-var-x}-\eqref{h-var-y} collapses into the classical version of the ADMM algorithm. 
\end{remark}

\begin{remark}\label{part-cases-alg-mon} For all $k \geq 0$, consider the particular choices $M_1^k:=\frac{1}{\tau_k}\id-cL^*L$ for $\tau_k > 0$, and $M_2^k:=0$.

(i) Let $k \geq 0$ be fixed. Relation \eqref{C-var-x} (written for $x^{k+2}$) reads
\begin{equation*}x^{k+2}=\left(\tau_{k+1}^{-1}\id+A\right)^{-1}\left[cL^*(z^{k+1}-c^{-1}y^{k+1})+\tau_{k+1}^{-1}x^{k+1}-cL^*Lx^{k+1}-Cx^{k+1}\right].\end{equation*}
From \eqref{C-var-y} we have $$cL^*(z^{k+1}-c^{-1}y^{k+1})=-L^*(2y^{k+1}-y^k)+cL^*L x^{k+1},$$
hence 
\begin{eqnarray}\label{x}x^{k+2}& = & \left(\tau_{k+1}^{-1}\id+A\right)^{-1}\left[\tau_{k+1}^{-1}x^{k+1}-L^*(2y^{k+1}-y^k)-Cx^{k+1}\right]\nonumber\\
& = & J_{\tau_{k+1}A}\left(x^{k+1}-\tau_{k+1}Cx^{k+1}-\tau_{k+1}L^*(2y^{k+1}-y^k)\right).
\end{eqnarray}

On the other hand, by using \eqref{j-inv-op}, relation \eqref{C-var-z} reads
\begin{eqnarray*}z^{k+1} = J_{c^{-1}B}\left(Lx^{k+1}+c^{-1}y^k\right) = Lx^{k+1}+c^{-1}y^k-c^{-1}J_{cB^{-1}}\left(cLx^{k+1}+y^k\right)
\end{eqnarray*}
which is equivalent to 
$$cLx^{k+1}+y^k-cz^{k+1}=J_{cB^{-1}}\left(cLx^{k+1}+y^k\right).$$
By using again \eqref{C-var-y}, this can be reformulated as
\begin{equation}\label{y} y^{k+1}=J_{cB^{-1}}\left(y^k+cLx^{k+1}\right).
\end{equation}

The iterative scheme in \eqref{x}- \eqref{y} generates for a given starting point $(x^1,y^0)\in{\cal H} \times{\cal G}$ and $c>0$ 
a sequence $(x^k,y^k)_{k \geq 1}$ which is generated for all $k \geq 0$ as follows:
\begin{eqnarray}
y^{k+1} & = & J_{cB^{-1}}\left(y^k+cLx^{k+1}\right)\label{vu1}\\
x^{k+2} & = & J_{\tau_{k+1}A}\left(x^{k+1}-\tau_{k+1}Cx^{k+1}-\tau_{k+1}L^*(2y^{k+1}-y^k)\right)\label{vu2}.
\end{eqnarray}

When $\tau_k=\tau$ for all $k\geq 1$, the algorithm \eqref{vu1}-\eqref{vu2} recovers a numerical scheme for solving monotone inclusion problems proposed by 
V\~{u} in \cite[Theorem 3.1]{vu}. More precisely, the error-free variant of the algorithm in \cite[Theorem 3.1]{vu} formulated for a constant sequence 
$(\lambda_n)_{n \in \N}$ equal to $1$ and employed to the solving of the primal-dual pair \eqref{d-mon-equiv}-\eqref{p-mon} (by reversing the order in 
Problem \ref{pr-mon}, that is, by treating \eqref{d-mon-equiv} as the primal monotone inclusion  and \eqref{p-mon} as its dual monotone inclusion) 
is nothing else than the iterative scheme \eqref{vu1}-\eqref{vu2}. 

In case $C=0$, \eqref{vu1}-\eqref{vu2} becomes for all $k \geq 0$
\begin{eqnarray}
x^{k+1} & = & J_{\tau_{k}A}\left(x^{k}-\tau_{k}L^*(2y^{k}-y^{k-1})\right)\label{bch1}\\
y^{k+1} & = & J_{cB^{-1}}\left(y^k+cLx^{k+1}\right)\label{bch2},
\end{eqnarray}
which, in case $\tau_k=\tau$ for all $k\geq 1$ and $c\tau\|L\|^2<1$, is nothing else than the algorithm introduced by Bo\c t, Csetnek and Heinrich  in \cite[Algorithm 1, Theorem 2]{b-c-h} 
applied to the solving of the primal-dual pair \eqref{d-mon-equiv}-\eqref{p-mon} (by reversing the order in Problem \ref{pr-mon}).

(ii) Considering again the variational setting as described in Problem \ref{pr-var}, the algorithm \eqref{vu1}-\eqref{vu2} reads for all $k \geq 0$
\begin{eqnarray}
y^{k+1} & = & \prox\nolimits_{cg^*}\left(y^k+cLx^{k+1}\right)\label{condat1}\\
x^{k+2} & = & \prox\nolimits_{\tau_{k+1}f}\left(x^{k+1}-\tau_{k+1}\nabla h(x^{k+1})-\tau_{k+1}L^*(2y^{k+1}-y^k)\right)\label{condat2}.
\end{eqnarray}
When $\tau_k=\tau >0$ for all $k\geq 1$, one recovers a primal-dual algorithm investigated 
under the assumption $\frac{1}{\tau}-c\|L\|^2>\frac{1}{2\eta}$ by Condat in \cite[Algorithm 3.2, Theorem 3.1]{condat2013}. 

Not least, \eqref{bch1}-\eqref{bch2} reads in the variational setting (which corresponds to the case $h=0$) for all $k \geq 0$
\begin{eqnarray}
x^{k+1} & = & \prox\nolimits_{\tau_k f}\left(x^{k}-\tau_{k}L^*(2y^{k}-y^{k-1})\right)\label{ch-pock1}\\
y^{k+1} & = & \prox\nolimits_{c g^*}\left(y^k+cLx^{k+1}\right)\label{ch-pock2}.
\end{eqnarray}
When $\tau_k=\tau >0$ for all $k\geq 1$, this iterative schemes becomes the algorithm proposed by Chambolle and Pock in \cite[Algorithm 1, Theorem 1]{ch-pck} for solving the primal-dual pair of 
optimization problems \eqref{dual-h}-\eqref{prim-h}  (in this order).
\end{remark}

\subsection{Convergence analysis}\label{sub22}

In this subsection we will address the convergence of the sequence of iterates generated by Algorithm \ref{alg-mon}. One of the  tools 
we will use in the proof of the convergence statement is  
the following version of the Opial Lemma formulated in the setting of variable metrics (see \cite[Theorem 3.3]{combettes-vu2013}). 

\begin{lemma}\label{opial-var} Let $S$ be a nonempty subset of ${\cal H}$ and $(x^k)_{k \geq 0}$ be a sequence in ${\cal H}$.  Let $\alpha>0$ and $W^k\in{\cal P}_{\alpha}({\cal H})$ be such that 
$W^k\succcurlyeq W^{k+1}$ for all $k\geq 0$. Assume that:  

(i) for all $z\in S$ and for all $k\geq 0$: $\|x^{k+1}-z\|_{W^{k+1}}\leq \|x^k-z\|_{W^k}$; 

(ii) every weak sequential cluster point of $(x^k)_{k\geq 0}$ belongs to $S$. 

\noindent Then $(x^k)_{k\geq 0}$ converges weakly to an element in $S$. 
\end{lemma}

We present the first main theorem of this manuscript. 

\begin{theorem}\label{cong-it} Consider the setting  of Problem \ref{pr-mon} and assume that the set of primal-dual solutions to the primal-dual pair of monotone inclusions \eqref{p-mon}-\eqref{d-mon} is nonempty. Let 
$(x^k,z^k,y^k)_{k\geq 0}$ be the sequence generated by Algorithm \ref{alg-mon} and assume that
$M_1^k-\frac{1}{2\eta}\id\in{\cal S_+}({\cal H}),
M_1^k\succcurlyeq M_1^{k+1}$, $M_2^k\in{\cal S_+}({\cal G}), M_2^k\succcurlyeq M_2^{k+1}$ for all $k\geq 0$.
If one of the following assumptions:
\begin{itemize}
\item[(I)] there exists $\alpha_1>0$ such that $M_1^k-\frac{1}{2\eta}\id\in{\cal P}_{\alpha_1}({\cal H})$ for all $k\geq 0$;

\item[(II)] there exist $\alpha, \alpha_2>0$ such that $L^*L\in {\cal P}_{\alpha}({\cal H})$ and $M_2^k\in {\cal P}_{\alpha_2}({\cal G})$ 
for all $k\geq 0$;
\end{itemize}
is fulfilled, then there exists $( x, v)$, a primal-dual solution to \eqref{p-mon}-\eqref{d-mon}, such that 
$(x^k,z^k,y^k)_{k\geq 0}$ converges weakly to $( x, Lx, v)$.  
\end{theorem}

\begin{proof} Let $S\subseteq {\cal H}\times {\cal G}\times {\cal G}$ be defined by 
\begin{equation}\label{def-s}S=\{(x,Lx,v):(x,v)\mbox{ is a primal dual solution to \eqref{p-mon}-\eqref{d-mon}}\}.\end{equation}

Let  $(x^*,Lx^*,y^*)\in S$ be fixed. Then it holds
$$-L^*y^*-Cx^*\in Ax^* \ \mbox{and} \  y^*\in B(Lx^*). $$
Let $k\geq 0$ be fixed. 
From \eqref{C-var-x-equiv} and the monotonicity of $A$ we have
\begin{equation}\label{mon-a}
\langle cL^*(z^k-Lx^{k+1}-c^{-1}y^k)+M_1^k(x^k-x^{k+1})-Cx^k+L^*y^*+Cx^*,x^{k+1}-x^*\rangle\geq 0, \end{equation}
while from \eqref{C-var-z-equiv} and the monotonicity of $B$ we have
\begin{equation}\label{mon-b}\langle c(Lx^{k+1}-z^{k+1}+c^{-1}y^k)+M_2^k(z^k-z^{k+1})-y^*,z^{k+1}-Lx^*\rangle\geq 0.\end{equation}
Since $C$ is $\eta$-cocoercive, we have  
$$\langle Cx^*-Cx^k,x^*-x^k\rangle\geq \eta\|Cx^*-Cx^k\|^2.$$

Summing up the three inequalities obtained above we get
\begin{align*}
c\langle z^k-Lx^{k+1},Lx^{k+1}-Lx^*\rangle+\langle y^*-y^k,Lx^{k+1}-Lx^*\rangle+\langle Cx^*-Cx^k,x^{k+1}-x^*\rangle &\\
+\langle M_1^k(x^k-x^{k+1}),x^{k+1}-x^*\rangle+c\langle Lx^{k+1}-z^{k+1},z^{k+1}-Lx^*\rangle+\langle y^k-y^*,z^{k+1}-Lx^*\rangle&\\
+\langle M_2^k(z^k-z^{k+1}),z^{k+1}-Lx^*\rangle+\langle Cx^*-Cx^k,x^*-x^k\rangle-\eta\|Cx^*-Cx^k\|^2 & \geq 0.
\end{align*}
According to \eqref{C-var-y} we also have
$$\langle y^*-y^k,Lx^{k+1}-Lx^*\rangle+\langle y^k-y^*,z^{k+1}-Lx^*\rangle=\langle y^*-y^k,Lx^{k+1}-z^{k+1}\rangle=
c^{-1}\langle y^*-y^k,y^{k+1}-y^k\rangle.$$
By expressing the inner products through norms we further derive
\begin{align*}
\frac{c}{2}\left(\|z^k-Lx^*\|^2-\|z^k-Lx^{k+1}\|^2-\|Lx^{k+1}-Lx^*\|^2\right) & \\
+ \frac{c}{2}\left(\|Lx^{k+1}-Lx^*\|^2-\|Lx^{k+1}-z^{k+1}\|^2-\|z^{k+1}-Lx^*\|^2\right) & \\
+\frac{1}{2c}\left(\|y^*-y^k\|^2+\|y^{k+1}-y^k\|^2-\|y^{k+1}-y^*\|^2\right)&\\
+\frac{1}{2}\left(\|x^k-x^*\|_{M_1^k}^2-\|x^k-x^{k+1}\|_{M_1^k}^2-\|x^{k+1}-x^*\|_{M_1^k}^2\right)&\\
+\frac{1}{2}\left(\|z^k-Lx^*\|_{M_2^k}^2-\|z^k-z^{k+1}\|_{M_2^k}^2-\|z^{k+1}-Lx^*\|_{M_2^k}^2\right)&\\
+\langle Cx^*- Cx^k,x^{k+1}-x^k\rangle-\eta\|Cx^*-Cx^k\|^2& \geq 0.
\end{align*}
By expressing $Lx^{k+1}-z^{k+1}$ using again relation \eqref{C-var-y} and by taking into account that
\begin{align*}
\langle Cx^*-Cx^k,x^{k+1}-x^k\rangle-\eta\|Cx^*-Cx^k\|^2 & =\\
-\eta^{-1}\left\|\eta\left(Cx^*-Cx^k\right)+\frac{1}{2}\left(x^k-x^{k+1}\right)\right\|^2+\frac{1}{4\eta}\|x^k-x^{k+1}\|^2&,
\end{align*}
we obtain 
\begin{align*}\frac{1}{2}\|x^{k+1}-x^*\|_{M_1^k}^2+\frac{1}{2}\|z^{k+1}-Lx^*\|_{M_2^k+c\id}^2+\frac{1}{2c}\|y^{k+1}-y^*\|^2 & \leq\\
\frac{1}{2}\|x^k-x^*\|_{M_1^k}^2+\frac{1}{2}\|z^k-Lx^*\|_{M_2^k+c\id}^2+\frac{1}{2c}\|y^k-y^*\|^2 & \\
-\frac{c}{2}\|z^k-Lx^{k+1}\|^2-\frac{1}{2}\|x^k-x^{k+1}\|_{M_1^k}^2-\frac{1}{2}\|z^k-z^{k+1}\|_{M_2^k}^2 &\\
-\eta^{-1}\left\|\eta\left(Cx^*-Cx^k\right)+\frac{1}{2}\left(x^k-x^{k+1}\right)\right\|^2+\frac{1}{4\eta}\|x^k-x^{k+1}\|^2.
\end{align*}
From here,  using the monotonicity assumptions on $(M_1^k)_{k \geq 0}$ and $(M_2^k)_{k \geq 0}$, it yields
\begin{align}\frac{1}{2}\|x^{k+1}-x^*\|_{M_1^{k+1}}^2+\frac{1}{2}\|z^{k+1}-Lx^*\|_{M_2^{k+1}+c\id}^2+\frac{1}{2c}\|y^{k+1}-y^*\|^2 & \leq \nonumber\\
\frac{1}{2}\|x^k-x^*\|_{M_1^k}^2+\frac{1}{2}\|z^k-Lx^*\|_{M_2^k+c\id}^2+\frac{1}{2c}\|y^k-y^*\|^2 & \nonumber \\
-\frac{c}{2}\|z^k-Lx^{k+1}\|^2-\frac{1}{2}\|x^k-x^{k+1}\|_{M_1^k-\frac{1}{2\eta}\id}^2-\frac{1}{2}\|z^k-z^{k+1}\|_{M_2^k}^2 & \nonumber\\
-\eta^{-1}\left\|\eta\left(Cx^*-Cx^k\right)+\frac{1}{2}\left(x^k-x^{k+1}\right)\right\|^2&.\label{fey-var}
\end{align}
Discarding the negative terms on the right-hand side of the above inequality 
(notice that $M_1^k-\frac{1}{2\eta}\id\in {\cal S_+}({\cal H})$ for all $k \geq 0$), 
it follows that statement (i) in Opial Lemma (Lemma \ref{opial-var}) holds, when applied in the product space 
${\cal H}\times {\cal G}\times {\cal G}$, for the sequence $(x^k,z^k,y^k)_{k\geq 0}$, for 
$W^k:= (M_1^k,M_2^k+c\id,c^{-1}\id)$ for $k \geq 0$, and for $S$ defined as in \eqref{def-s}.

Furthermore, summing up the inequalities in \eqref{fey-var}, we get  
\begin{equation}\label{series}\sum_{k\geq 0}\|z^k-Lx^{k+1}\|^2<+\infty, \ \sum_{k\geq 0}\|x^k-x^{k+1}\|_{M_1^k-\frac{1}{2\eta}\id}^2<+\infty, \ \sum_{k\geq 0}\|z^k-z^{k+1}\|_{M_2^k}^2<+\infty.\end{equation}

Consider first the hypotheses in assumption (I). Since $M_1^k-\frac{1}{2\eta}\id\in{\cal P}_{\alpha_1}({\cal H})$ for all $k \geq 0$ with $\alpha_1>0$, we get 
\begin{equation}\label{x-x-0} x^k-x^{k+1}\rightarrow 0 \ (k\rightarrow+\infty)
\end{equation}
and \begin{equation}\label{z-Ax-0} z^k-Lx^{k+1}\rightarrow 0 \ (k\rightarrow+\infty).
\end{equation}
A direct consequence of \eqref{x-x-0} and \eqref{z-Ax-0} is 
\begin{equation}\label{z-z-0} z^k-z^{k+1}\rightarrow 0 \ (k\rightarrow+\infty). 
\end{equation}
From \eqref{C-var-y}, \eqref{z-Ax-0} and \eqref{z-z-0} we derive 
\begin{equation}\label{y-y-0} y^k-y^{k+1}\rightarrow 0 \ (k\rightarrow+\infty). 
\end{equation}

Next we show that the relations \eqref{x-x-0}-\eqref{y-y-0} are fulfilled also under assumption (II). Indeed, in this situation we derive 
from \eqref{series} that \eqref{z-Ax-0} and \eqref{z-z-0} hold. 
From \eqref{C-var-y}, \eqref{z-Ax-0} and \eqref{z-z-0} we obtain \eqref{y-y-0}. Finally, 
the inequalities 
\begin{equation}\label{ineq}\alpha\|x^{k+1}-x^k\|^2\leq \|Lx^{k+1}-Lx^k\|^2\leq 2\|Lx^{k+1}-z^k\|^2+2\|z^k-Lx^k\|^2 \ \forall k\geq 0\end{equation}
yield \eqref{x-x-0}.

The relations \eqref{x-x-0}-\eqref{y-y-0} will play an essential role when verifying  assumption (ii) in the Opial Lemma for 
$S$ taken as in \eqref{def-s}. Let $(\ol x,\ol z,\ol y)\in {\cal H}\times {\cal G}\times {\cal G}$ be such that 
there exists $(k_n)_{n\geq 0}$, $k_n\rightarrow +\infty$ (as $n\rightarrow +\infty$), and $(x^{k_n},z^{k_n}, y^{k_n})$ converges weakly to 
$(\ol x,\ol z,\ol y)$ (as $n\rightarrow +\infty$). 

From \eqref{x-x-0} and the linearity and the continuity of $L$ we obtain that $(Lx^{k_n+1})_{n\in\N}$ converges weakly 
to $L\ol x$ (as $n\rightarrow +\infty$), which combined with \eqref{z-Ax-0} yields $\ol z=L\ol x$. 
We use now the following notations for $n\geq 0$:
\begin{align*} 
a_n^* := & \ cL^*(z^{k_n}-Lx^{k_n+1}-c^{-1}y^{k_n})+M_1^{k_n}(x^{k_n}-x^{k_n+1})+Cx^{k_n+1}-Cx^{k_n}\\
a_n := & \ x^{k_n+1}\\
b_n^*:= &\  y^{k_n+1}+M_2^{k_n}(z^{k_n}-z^{k_n+1})\\
b_n:= & \ z^{k_n+1}.
\end{align*}
From \eqref{C-var-x-equiv} we have for all $n\geq 0$
\begin{equation}\label{a-a}a_n^*\in (A+C)(a_n).\end{equation}
Further, from \eqref{C-var-z-equiv} and \eqref{C-var-y} we have for all $n\geq 0$
\begin{equation}\label{b-b}b_n^*\in Bb_n.\end{equation}
Furthermore, from \eqref{x-x-0} we have 
\begin{equation}\label{an-w}a_n \mbox{ converges weakly to } \ol x \ (\mbox{as }n\rightarrow+\infty). 
\end{equation}
From \eqref{y-y-0} and \eqref{z-z-0} we obtain 
\begin{equation}\label{bn-w}b_n^* \mbox{ converges weakly to } \ol y \ (\mbox{as }n\rightarrow+\infty). 
\end{equation}
Moreover, \eqref{C-var-y} and \eqref{y-y-0} yield 
\begin{equation}\label{Aan-bn}La_n-b_n \mbox{ converges strongly to } 0 \ (\mbox{as }n\rightarrow+\infty). 
\end{equation}
Finally, we have 
\begin{align*}
a_n^*+L^*b_n^*= & \ cL^*(z^{k_n}-Lx^{k_n+1})+L^*(y^{k_n+1}-y^{k_n})\\
& +M_1^{k_n}(x^{k_n}-x^{k_n+1}) +L^*M_2^{k_n}(z^{k_n}-z^{k_n+1})\\
& +Cx^{k_n+1}-Cx^{k_n}.
\end{align*}
By using the fact that $C$ is $\eta^{-1}$-Lipschitz continuous, from \eqref{x-x-0}-\eqref{y-y-0} we get 
\begin{equation}\label{an-Abn}a_n^*+L^*b_n^* \mbox{ converges strongly to } 0 \ (\mbox{as }n\rightarrow+\infty). 
\end{equation}
Let us define $T:{\cal H}\times {\cal G}\rightrightarrows {\cal H}\times {\cal G}$ by $T(x,y)=(A(x)+C(x)) \times B^{-1}(y)$ and 
$K:{\cal H}\times {\cal G}\rightarrow {\cal H}\times {\cal G}$ by $K(x,y)=(L^*y,-Lx)$ for all 
$(x,y)\in{\cal H}\times {\cal G}$. Since $C$ is maximally monotone with full domain (see \cite{bauschke-book}), 
$A+C$ is maximally monotone, too (see \cite{bauschke-book}), thus $T$ is maximally monotone. Since 
$K$ is s skew operator, it is also maximally monotone (see \cite{bauschke-book}). Due to the fact that $K$ has full domain, we conclude 
that \begin{equation}\label{t-plus-k}T+K \mbox{ is a maximally monotone operator}.\end{equation}
Moreover, from \eqref{a-a} and \eqref{b-b} we have 
\begin{equation}\label{seq-prod-sp}(a_n^*+L^*b_n^*,b_n-La_n)\in (T+K)(a_n,b_n^*) \ \forall n\geq 0.\end{equation}
Since the graph of a maximally monotone operator is sequentially closed with respect to the weak$\times$strong topology 
(see \cite[Proposition 20.33]{bauschke-book}), from \eqref{t-plus-k}, \eqref{seq-prod-sp}, \eqref{an-w}, \eqref{bn-w}, \eqref{Aan-bn} 
and \eqref{an-Abn} we derive that 
$$(0,0)\in (T+K)(\ol x,\ol y)=(A+C,B^{-1})(\ol x,\ol y)+(L^*\ol y,-L\ol x).$$
The latter is nothing else than saying that $(\ol x,\ol y)$ is a primal dual-solution to \eqref{p-mon}-\eqref{d-mon}, which combined 
with $\ol z=L\ol x$ implies that the second assumption of the Opial Lemma is verified, too. 
In conclusion, $(x^k,z^k,y^k)_{k\geq 0}$ converges weakly to $(x,Lx,v)$, where $(x,v)$ a primal-dual solution to \eqref{p-mon}-\eqref{d-mon}.
\end{proof}

\begin{remark}\label{rem-th1} (i) Choosing as in Remark \ref{part-cases-alg-mon} $M_1^k:= \frac{1}{\tau_k}\id - cL^*L$, with $\tau_k >0$ and 
$\tau:=\sup_{k \geq 0} \tau_k \in \R $, and $M_2^k := 0$ for all $k \geq 0$, we have
$$\left\langle x,\left(M_1^k-\frac{1}{2\eta}\id\right)x\right\rangle\geq\left(\frac{1}{\tau_k}-c\|L\|^2-\frac{1}{2\eta}\right)\|x\|^2  
\geq \left(\frac{1}{\tau}-c\|L\|^2-\frac{1}{2\eta}\right)\|x\|^2 \ \forall x \in {\cal H},$$
which means that under the assumption $\frac{1}{\tau}-c\|L\|^2>\frac{1}{2\eta}$ 
(which recovers the one in Algorithm 3.2 and Theorem 3.1 in \cite{condat2013}) the operators $M_1^k-\frac{1}{2\eta}\id$ 
belong for all $k \geq 0$ to the class ${\cal P}_{\alpha_1}({\cal H})$ with $\alpha_1:= \frac{1}{\tau}-c\|L\|^2-\frac{1}{2\eta}>0$. 

(ii) Let us briefly discuss the condition considered in (II): 
\begin{equation}\label{h}\exists \alpha>0 \mbox{ such that } L^*L\in {\cal P}_{\alpha}({\cal H}). 
\end{equation}
By taking into account \cite[Fact 2.19]{bauschke-book}, one can see that \eqref{h} holds if and only if 
$L$ is injective and $\ran L^*$ is closed. This means that 
if $\ran L^*$ is closed, then \eqref{h} is equivalent to $L$ is injective. Hence, in finite dimensional spaces, namely, if ${\cal H}=\R^n$ and 
${\cal G}=\R^m$, with $m\geq n\geq 1$, \eqref{h}  is nothing else than saying that $L$ has full column rank, which is a widely used assumption 
in the proof of the convergence of the classical ADMM algorithm. 
\end{remark}

In the second convergence result of this section we consider the case when $C$ is identically $0$. We notice that this cannot 
be encompassed in the above theorem due to the assumptions which involve 
the cococercivity constant $\eta$ in the denominator of some fractions and which do not allow us to take it equal to zero.

\begin{theorem}\label{cong-it-sh-teb} Consider the setting of Problem \ref{pr-mon} in the situation when $C=0$ and assume that the set of primal-dual solutions to the primal-dual pair of monotone inclusions 
\eqref{p-mon}-\eqref{d-mon} is nonempty. Let $(x^k,z^k,y^k)_{k\geq 0}$ be the sequence generated by Algorithm \ref{alg-mon} and assume that $M_1^k\in{\cal S_+}({\cal H}),
M_1^k\succcurlyeq M_1^{k+1}$, $M_2^k\in{\cal S_+}({\cal G}), M_2^k\succcurlyeq M_2^{k+1}$ for all $k\geq 0$. 
If one of the following assumptions:
\begin{itemize}
\item[(I)] there exists $\alpha_1>0$ such that $M_1^k\in{\cal P}_{\alpha_1}({\cal H})$ for all $k\geq 0$;

\item[(II)] there exist $\alpha, \alpha_2>0$ such that $L^*L\in {\cal P}_{\alpha}({\cal H})$ and 
$M_2^k\in {\cal P}_{\alpha_2}({\cal G})$ for all $k\geq 0$; 

\item[(III)] there exists $\alpha>0$ such that $L^*L\in {\cal P}_{\alpha}({\cal H})$ and 
$2M_2^{k+1}\succcurlyeq M_2^{k}\succcurlyeq M_2^{k+1}$ for all $k\geq 0$;
\end{itemize}
is fulfilled, then there exists $(x, v)$, a primal-dual solution to \eqref{p-mon}-\eqref{d-mon}, such that 
$(x^k,z^k,y^k)_{k\geq 0}$ converges weakly to $(x, Lx, v)$.   
\end{theorem}

\begin{proof}  We fix an element $(x^*,Lx^*,y^*)$ with the property that $(x^*, y^*)$ a primal-dual solution to 
\eqref{p-mon}-\eqref{d-mon}. 

Take an arbitrary $k \geq 0$. 
As in the proof of Theorem \ref{cong-it}, we derive the inequality 
\begin{align}\frac{1}{2}\|x^{k+1}-x^*\|_{M_1^{k+1}}^2+\frac{1}{2}\|z^{k+1}-Lx^*\|_{M_2^{k+1}+c\id}^2+\frac{1}{2c}\|y^{k+1}-y^*\|^2 & \leq \nonumber\\
 \frac{1}{2}\|x^k-x^*\|_{M_1^k}^2+\frac{1}{2}\|z^k-Lx^*\|_{M_2^k+c\id}^2+\frac{1}{2c}\|y^k-y^*\|^2 &\nonumber\\
 -\frac{c}{2}\|z^k-Lx^{k+1}\|^2-\frac{1}{2}\|x^k-x^{k+1}\|_{M_1^k}^2-\frac{1}{2}\|z^k-z^{k+1}\|_{M_2^k}^2.& \label{fey-var-h0}
\end{align}
Under assumption (I) the conclusion follows as in the proof of Theorem \ref{cong-it} by making use of the Opial Lemma. 

Consider the situation when the hypotheses in assumption (II) are fulfilled. 

By using telescopic sum techniques, it follows that \eqref{z-Ax-0} and \eqref{z-z-0} hold. From \eqref{C-var-y}, \eqref{z-Ax-0} and \eqref{z-z-0} we obtain 
\eqref{y-y-0}. Finally, by using again the inequality \eqref{ineq}, relation \eqref{x-x-0} holds, too. 

On the other hand, \eqref{fey-var-h0} yields that 
\begin{equation}\label{exist-lim}  \exists\lim_{k\rightarrow+\infty}\left(\frac{1}{2}\|x^k-x^*\|_{M_1^k}^2+
\frac{1}{2}\|z^k-Lx^*\|_{M_2^k+c\id}^2+\frac{1}{2c}\|y^k-y^*\|^2\right),
\end{equation}
hence $(y^k)_{k\geq 0}$ and $(z^k)_{k\geq 0}$ are bounded. Combining this with the condition 
imposed on $L$, we derive that $(x^k)_{k\geq 0}$ is bounded, too. Hence there exists a weak convergent subsequence 
of $(x^k,z^k,y^k)_{k\geq 0}$. By using the same arguments as in the second part of the proof of Theorem \ref{cong-it}, one can see that 
every sequential weak cluster point of $(x^k,z^k,y^k)_{k\geq 0}$ belongs to the set $S$ defined in \eqref{def-s}. 

In the remaining of the proof we show that the set of sequential weak cluster points of $(x^k,z^k,y^k)_{k\geq 0}$ is a singleton. 
Let $(x_1,z_1,y_1),(x_2,z_2,y_2)$ be two such sequential weak cluster points. Then there exist $(k_p)_{p\geq 0}, (k_q)_{q\geq 0}$, 
$k_p\rightarrow+\infty$ (as $p\rightarrow+\infty$), $k_q\rightarrow+\infty$ (as $q\rightarrow+\infty$), a subsequence
$(x^{k_p},z^{k_p}, y^{k_p})_{p \geq 0}$ which converges weakly to $(x_1,z_1,y_1)$ (as $p\rightarrow+\infty$), and a subsequence
$(x^{k_q},z^{k_q}, y^{k_q})_{q \geq 0}$ which converges weakly to $(x_2,z_2,y_2)$ (as $q\rightarrow+\infty$). As shown above, 
$(x_1,z_1,y_1)$ and 
$(x_2,z_2,y_2)$ belong to the set $S$ (see \eqref{def-s}), thus $z_i=Lx_i$, $i\in\{1,2\}$. From \eqref{exist-lim}, which is true 
for every primal-dual solution to \eqref{p-mon}-\eqref{d-mon}, we derive 
\begin{equation}\label{exist-lim1}
\exists\lim_{k\rightarrow+\infty}\left(E(x^k,z^k,y^k; x_1,Lx_1,y_1)-E(x^k,z^k,y^k; x_2,Lx_2,y_2)\right),
\end{equation}
where, for $(x^*,Lx^*,y^*)$ the expression $E(x^k,z^k,y^k; x^*,Lx^*,y^*)$ is defined as 
$$E(x^k,z^k,y^k; x^*,Lx^*,y^*)=\frac{1}{2}\|x^k-x^*\|_{M_1^k}^2+
\frac{1}{2}\|z^k-Lx^*\|_{M_2^k+c\id}^2+\frac{1}{2c}\|y^k-y^*\|^2.$$
Further, we have for all $k \geq 0$
$$\frac{1}{2}\|x^k-x_1\|_{M_1^k}^2-\frac{1}{2}\|x^k-x_2\|_{M_1^k}^2=\frac{1}{2}\|x_2-x_1\|_{M_1^k}^2
+\langle x^k-x_2,M_1^k(x_2-x_1)\rangle,$$
$$\frac{1}{2}\|z^k-Lx_1\|_{M_2^k+c\id}^2\!\!-\frac{1}{2}\|z^k-Lx_2\|_{M_2^k+c\id}^2 \!\!=\!\!\frac{1}{2}\|Lx_2-Lx_1\|_{M_2^k+c\id}^2
\!\!+\langle z^k-Lx_2, (M_2^k+c\id)(Lx_2-Lx_1)\rangle,$$
and
$$\frac{1}{2c}\|y^k-y_1\|^2-\frac{1}{2c}\|y^k-y_2\|^2=\frac{1}{2c}\|y_2-y_1\|^2
+\frac{1}{c}\langle y^k-y_2, y_2-y_1\rangle.$$

Applying \cite[Th\'{e}or\`{e}ème 104.1]{rn}, there exists $M_1\in{\cal S_+}({\cal H})$ such that 
$(M_1^k)_{k \geq 0}$ converges pointwise to $M_1$ in the strong topology (as $k\rightarrow+\infty$). Similarly, 
the monotonicity condition imposed on $(M_2^k)_{k \geq 0}$ implies that $\sup_{k\geq 0}\|M_2^k+c\id\|<+\infty$. 
Thus, according to \cite[Lemma 2.3]{combettes-vu2013}, there exists $\alpha'>0$  and $M_2\in {\cal P}_{\alpha'}({\cal G})$ 
such that $(M_2^k+c\id)_{k \geq 0}$ converges pointwise to $M_2$ in the strong topology (as $k\rightarrow+\infty$). 

Taking the limit in \eqref{exist-lim1} along the subsequences $(k_p)_{p\geq 0}$ and $(k_q)_{q\geq 0}$ and using the last three relations 
above we obtain 
$$\frac{1}{2}\|x_1-x_2\|_{M_1}^2+\langle x_1-x_2,M_1(x_2-x_1)\rangle+ \frac{1}{2}\|Lx_1-Lx_2\|_{M_2}^2+\langle Lx_1-Lx_2, M_2(Lx_2-Lx_1)\rangle$$
$$+\frac{1}{2c}\|y_1-y_2\|^2+\frac{1}{c}\langle y_1-y_2, y_2-y_1\rangle
=\frac{1}{2}\|x_1-x_2\|_{M_1}^2+\frac{1}{2}\|Lx_1-Lx_2\|_{M_2}^2+\frac{1}{2c}\|y_1-y_2\|^2,$$
hence $$-\|x_1-x_2\|_{M_1}^2-\|Lx_1-Lx_2\|_{M_2}^2-\frac{1}{c}\|y_1-y_2\|^2=0,$$
thus $Lx_1=Lx_2$ and $y_1=y_2$. The condition on $L$ imposed in Assumption (II) implies that $x_1=x_2$.  
In conclusion, $(x^k,z^k,y^k)_{k\geq 0}$ converges weakly to an element in $S$ (see\eqref{def-s}). 

Finally, we consider the situation when the hypotheses in assumption (III) hold. 

As noticed above, relation 
\eqref{fey-var-h0} holds. Let $k\geq 1$ be fixed. By considering the relation \eqref{C-var-z-equiv-equiv-equiv} for consecutive iterates 
and by taking into account the monotonicity of $B$ we derive 
$$\langle z^{k+1}-z^k,y^{k+1}-y^k+M_2^k(z^k-z^{k+1})-M_2^{k-1}(z^{k-1}-z^k)\rangle\geq 0,$$
hence 
\begin{align*}
\langle z^{k+1}-z^k,y^{k+1}-y^k\rangle & \geq \|z^{k+1}-z^k\|_{M_2^k}^2+\langle z^{k+1}-z^k,M_2^{k-1}(z^{k-1}-z^k)\rangle\\
& \geq \|z^{k+1}-z^k\|_{M_2^k}^2-\frac{1}{2}\|z^{k+1}-z^k\|_{M_2^{k-1}}^2-\frac{1}{2}\|z^{k}-z^{k-1}\|_{M_2^{k-1}}^2.
\end{align*}
Substituting $y^{k+1}-y^k=c(Lx^{k+1}-z^{k+1})$ in the last inequality it follows
\begin{align*}
 \|z^{k+1}-z^k\|_{M_2^k}^2-\frac{1}{2}\|z^{k+1}-z^k\|_{M_2^{k-1}}^2-\frac{1}{2}\|z^{k}-z^{k-1}\|_{M_2^{k-1}}^2 & \leq\\
 \frac{c}{2}\left(\|z^k-Lx^{k+1}\|^2-\|z^{k+1}-z^k\|^2-\|Lx^{k+1}-z^{k+1}\|^2\right),
\end{align*}
which, after adding it with \eqref{fey-var-h0}, leads to
\begin{align}\frac{1}{2}\|x^{k+1}-x^*\|_{M_1^{k+1}}^2+\frac{1}{2}\|z^{k+1}-Lx^*\|_{M_2^{k+1}+c\id}^2+\frac{1}{2c}\|y^{k+1}-y^*\|^2 +
\frac{1}{2}\|z^{k+1}-z^k\|_{3M_2^k-M_2^{k-1}}^2& \!\!\leq \nonumber\\
 \frac{1}{2}\|x^{k}-x^*\|_{M_1^{k}}^2+\frac{1}{2}\|z^k-Lx^*\|_{M_2^k+c\id}^2+\frac{1}{2c}\|y^k-y^*\|^2+\frac{1}{2}\|z^k-z^{k-1}\|_{M_2^{k-1}}^2&\nonumber\\
 -\frac{1}{2}\|x^{k+1}-x^k\|_{M_1^k}^2-\frac{c}{2}\|z^{k+1}-z^k\|^2 - \frac{1}{2c}\|y^{k+1}-y^k\|^2.& \label{fey-var-h0-M10-III}
\end{align}
Taking into account that according to (III) we have $3M_2^k-M_2^{k-1}\succcurlyeq M_2^k$, we can conclude that for all $k \geq 1$ it holds
\begin{align}\frac{1}{2}\|x^{k+1}-x^*\|_{M_1^{k+1}}^2+\frac{1}{2}\|z^{k+1}-Lx^*\|_{M_2^{k+1}+c\id}^2+\frac{1}{2c}\|y^{k+1}-y^*\|^2 +
\frac{1}{2}\|z^{k+1}-z^k\|_{M_2^k}^2& \!\!\leq \nonumber\\
 \frac{1}{2}\|x^{k}-x^*\|_{M_1^{k}}^2+\frac{1}{2}\|z^k-Lx^*\|_{M_2^k+c\id}^2+\frac{1}{2c}\|y^k-y^*\|^2+\frac{1}{2}\|z^k-z^{k-1}\|_{M_2^{k-1}}^2&\nonumber\\
 -\frac{1}{2}\|x^{k+1}-x^k\|_{M_1^k}^2-\frac{c}{2}\|z^{k+1}-z^k\|^2 - \frac{1}{2c}\|y^{k+1}-y^k\|^2.& \label{fey-var-h0-M10-III'}
\end{align}
Using telescoping sum arguments, we obtain that $y^k - y^{k+1} \rightarrow 0$ and $z^k - z^{k+1} \rightarrow 0$ as $k \rightarrow +\infty$. Using \eqref{C-var-y}, it follows that
$L(x^{k}-x^{k+1}) \rightarrow 0$ as $k \rightarrow +\infty$, which, combined with the hypotheses imposed to $L$, further implies that $x^{k}-x^{k+1} \rightarrow 0$ as $k \rightarrow +\infty$.
Consequently, $z^k - Lx^{k+1} \rightarrow 0$ as $k \rightarrow +\infty$. Hence the relations \eqref{x-x-0}-\eqref{y-y-0} are fulfilled. 
On the other hand, from \eqref{fey-var-h0-M10-III'} we also derive that 
\begin{equation}\label{exist-lim'} 
\exists\lim_{k\rightarrow+\infty}\left(\frac{1}{2}\|x^{k}-x^*\|_{M_1^{k}}^2+\frac{1}{2}\|z^k-Lx^*\|_{M_2^k+c\id}^2+\frac{1}{2c}\|y^k-y^*\|^2+\frac{1}{2}\|z^k-z^{k-1}\|_{M_2^{k-1}}^2\right).
\end{equation}
By using that 
$$\|z^k-z^{k-1}\|_{M_2^{k-1}}^2 \leq \|z^k-z^{k-1}\|_{M_2^{0}}^2 \leq \|M_2^0\| \|z^k-z^{k-1}\|^2 \ \forall k \geq 1,$$
it follows that $\lim_{k\rightarrow+\infty} \|z^k-z^{k-1}\|_{M_2^{k-1}}^2 = 0$, which further implies that \eqref{exist-lim} holds. From here the conclusion follows by arguing as in the second part of 
the proof provided in the case when assumption (II) is fulfilled.
\end{proof} 

\begin{remark} In the finite dimensional variational case, the sequences generated by the classical ADMM algorithm, which corresponds to the iterative 
scheme \eqref{h-var-x}-\eqref{h-var-y} for $h=0$ and $M_1^k=M_2^k=0$ for all $k\geq 0$, are convergent, provided that $L$ has full column rank. 
This situation is covered by the theorem above in the context of assumption (III).
\end{remark}

\section{Convergence rates under strong monotonicity and by means of dynamic step sizes}\label{sec3}

We state the problem on which we focus throughout this section. 

\begin{problem}\label{pr-acc} In the setting of Problem \ref{pr-mon} we replace the cocoercivity of $C$ by the assumptions that
$C$ is monotone and $\mu$-Lipschitz continuous for $\mu\geq 0$. Moreover, we assume that $A+C$ is 
$\gamma$-strongly monotone for $\gamma >0$. 
\end{problem}

\begin{remark}\label{acc1-cocoerc-Lip}
If $C$  is a $\eta$-cocoercive operator for $\eta >0$, then $C$ is monotone and $\eta^{-1}$-Lipschitz continuous. Though, the converse statement may fail. 
The skew operator $(x,y)\mapsto (L^*y,-Lx)$ is for instance monotone and Lipschitz continuous, and not cocoercive. This operator appears in a natural way when considering 
formulating the system of optimality conditions for convex optimization problems involving compositions with linear continuous operators (see \cite{br-combettes}). 
Notice that due to the celebrated Baillon-Haddad Theorem (see, for instance, 
\cite[Corollary 8.16]{bauschke-book}), the gradient of a convex and Fr\'echet differentiable function is $\eta$-cocoercive if and only if it is $\eta^{-1}$-Lipschitz continuous.
\end{remark}

\begin{remark} In the setting of Problem \ref{pr-acc} the operator $A+L^*\circ B\circ L+C$ is strongly monotone, thus the monotone inclusion problem \eqref{p-mon} has at most one solution. 
Hence, if $(x,v)$ is a primal-dual solution to the primal-dual pair \eqref{p-mon}-\eqref{d-mon}, 
then $x$ is the unique solution to \eqref{p-mon}. Notice that the problem \eqref{d-mon} 
may not have an unique solution.
\end{remark}

We propose the following algorithm for the formulation of which we use dynamic step sizes.

\begin{algorithm}\label{alg-mon-acc} For all $k \geq 0$, let $M_2^k:{\cal G}\to{\cal G}$ be a linear, continuous and self-adjoint operator  such 
that $\tau_k LL^*+M_2^k\in{\cal P}_{\alpha_k}({\cal G})$ for $\alpha_k>0$ for all $k\geq 0$. 
Choose $(x^0,z^0,y^0)\in{\cal H}\times{\cal G}\times{\cal G}$. For all $k\geq 0$ generate the sequence $(x^k,z^k,y^k)_{k \geq 0}$ as follows:
\begin{align} y^{k+1} = & \
\left(\tau_k LL^*+M_2^k+B^{-1}\right)^{-1}\left[-\tau_k L(z^k-\tau_k^{-1}x^k)+M_2^ky^k\right] \label{C-var-y-acc}\\
 z^{k+1} = & \ \left(\frac{\theta_k}{\lambda}-1\right)L^*y^{k+1} + \frac{\theta_k}{\lambda}Cx^k+ \frac{\theta_k}{\lambda}\left(\id+\lambda\tau_{k+1}^{-1}A^{-1}\right)^{-1}
\left[-L^*y^{k+1}+\lambda\tau_{k+1}^{-1}x^k-Cx^k\right] \label{C-var-z-acc}\\
x^{k+1} = & \ x^k+\frac{\tau_{k+1}}{\theta_k}\left(-L^*y^{k+1}-z^{k+1}\right),\label{C-var-x-acc}
\end{align}
where $\lambda,\tau_k, \theta_k>0$ for all $k\geq 0$.
\end{algorithm}

\begin{remark} We would like to emphasize that when $C=0$ Algorithm \ref{alg-mon-acc} has a similar structure to Algorithm \ref{alg-mon}. Indeed, in this setting, the monotone 
inclusion problems \eqref{p-mon} and \eqref{d-mon-equiv} become 
\begin{equation}\label{p-mon-c0}
\mbox{find } x \in {\cal H} \ \mbox{such that} \ 0\in A x+ (L^*\circ B\circ L) x
\end{equation}
and, respectively,
\begin{equation}\label{d-mon-equiv-c0}
\mbox{ find } v\in{\cal G} \ \mbox{such that }   \ 0\in B^{-1} v+ \Big((-L)\circ (A^{-1})\circ (-L^*)\Big) v.
\end{equation}
The two problems \eqref{p-mon-c0} and \eqref{d-mon-equiv-c0} are dual to each other in the sense of the Attouch-Th\'{e}ra duality (see \cite{att-th}). By taking in 
\eqref{C-var-y-acc}-\eqref{C-var-x-acc} $\lambda=1$, $\theta_k=1$ (which corresponds to the limit case $\mu=0$ and $\gamma=0$ in the equation \eqref{theta} below) and 
$\tau_k=c > 0$ for all $k\geq 0$,  then the resulting iterative scheme
reads
\begin{eqnarray} y^{k+1} & = &
\left(c LL^*+M_2^k+B^{-1}\right)^{-1}\left[-c L(z^k-c^{-1}x^k)+M_2^ky^k\right] \label{C-var-y-acc-c0}\\
z^{k+1} & = &\left(\id+c^{-1}A^{-1}\right)^{-1}
\left[-L^*y^{k+1}+c^{-1}x^k\right] \label{C-var-z-acc-c0}\\
x^{k+1} & = & x^k+c\left(-L^*y^{k+1}-z^{k+1}\right).\label{C-var-x-acc-c0}
\end{eqnarray}
This is nothing else than Algorithm \ref{alg-mon} employed to the solving of the primal-dual system of monotone inclusions 
\eqref{d-mon-equiv-c0}-\eqref{p-mon-c0}, that is, by treating  \eqref{d-mon-equiv-c0} as the primal monotone inclusion  and 
\eqref{p-mon-c0} as its dual monotone inclusion
(notice that in this case we take in relation \eqref{C-var-z} of Algorithm \ref{alg-mon} $M_2^k=0$ for all $k\geq 0$).    
\end{remark}

Concerning the parameters involved in Algorithm \ref{alg-mon-acc} we assume that 
\begin{equation}\label{tau1} \mu\tau_1< 2\gamma,
\end{equation}
\begin{equation}\label{lambda} \lambda\geq \mu+1, 
\end{equation}  
there exists $\sigma_0>0$ such that 
\begin{equation}\label{sigma0-tau1} \sigma_0\tau_1\|L\|^2\leq 1,
\end{equation}
and for all $k\geq 0$:

\begin{equation}\label{theta} \theta_k =\frac{1}{\sqrt{1+\tau_{k+1}\lambda^{-1}(2\gamma-\mu\tau_{k+1})}}
\end{equation}
\begin{equation}\label{tau} \tau_{k+2}=\theta_{k}\tau_{k+1}
\end{equation}
\begin{equation}\label{sigma} \sigma_{k+1}=\theta_{k}^{-1}\sigma_{k}
\end{equation}
\begin{equation}\label{mon1} \tau_k LL^*+M_2^k\succcurlyeq \sigma_k^{-1}\id 
\end{equation}
\begin{equation}\label{mon2} \frac{\tau_k}{\tau_{k+1}} LL^*+\frac{1}{\tau_{k+1}}M_2^k\succcurlyeq 
\frac{\tau_{k+1}}{\tau_{k+2}} LL^*+\frac{1}{\tau_{k+2}}M_2^{k+1}.
\end{equation}

\begin{remark}\label{rem-equiv-acc} Fix an arbitrary $k\geq 1$. From \eqref{C-var-y-acc} we have 
\begin{equation}\label{c1} -\tau_k L(z^k-\tau_k^{-1}x^k)+M_2^ky^k\in \widetilde{M_2}^ky^{k+1}+B^{-1}y^{k+1},\end{equation}
where 
\begin{equation}\label{def-m_tild}  \widetilde{M_2}^k:= \tau_k LL^*+M_2^k.
\end{equation}
Due to \eqref{C-var-x-acc} we have $$-\tau_kz^k=\tau_k L^*y^k+\theta_{k-1}(x^k-x^{k-1}),$$
which combined with \eqref{c1} delivers 
\begin{equation}\label{C-var-y-acc-equiv} \widetilde{M_2}^k(y^k-y^{k+1})+L\left[x^k+\theta_{k-1}(x^k-x^{k-1})\right]\in B^{-1}y^{k+1}.
\end{equation}
Fix  now an arbitrary $k\geq 0$. From \eqref{j-inv-op} and \eqref{C-var-z-acc} we have 
\begin{align*}-L^*y^{k+1}+\frac{\lambda}{\theta_k}\left(z^{k+1}+L^*y^{k+1}\right)-Cx^k =&-L^*y^{k+1}+\frac{\lambda}{\tau_{k+1}}x^k-Cx^k\\
&-\frac{\lambda}{\tau_{k+1}}J_{(\tau_{k+1}/\lambda)A}\left[x^k+\frac{\tau_{k+1}}{\lambda}\left(-L^*y^{k+1}-Cx^k\right)\right].
\end{align*}
By using \eqref{C-var-x-acc} we obtain 
\begin{equation}\label{C-var-x-acc-equiv}
x^{k+1} =J_{(\tau_{k+1}/\lambda)A}\left[x^k+\frac{\tau_{k+1}}{\lambda}\left(-L^*y^{k+1}-Cx^k\right)\right].
\end{equation}
Finally, the definition of the resolvent yields the relation
\begin{equation}\label{C-var-x-acc-equiv-def-res}
\frac{\lambda}{\tau_{k+1}}\left(x^k-x^{k+1}\right)-L^*y^{k+1}+Cx^{k+1}-Cx^k\in(A+C)x^{k+1}.
\end{equation}
\end{remark}

\begin{remark}\label{param-choice-pf-acc} The choice 
\begin{equation}\label{choice-pd} \tau_k LL^*+M_2^k=\sigma_k^{-1}\id \ \forall k\geq 0
\end{equation}
leads to so-called accelerated versions of primal-dual algorithms that have been intensively studied in the literature. Indeed, under these auspices \eqref{C-var-y-acc} becomes 
(by taking into account also \eqref{C-var-x-acc})
\begin{eqnarray*} y^{k+1} &=& \left(\sigma_k^{-1}\id+B^{-1}\right)^{-1}\left[L\left(-\tau_kz^k+x^k-\tau_kL^*y^k\right)+\sigma_k^{-1}y^k\right]\\ 
&=& \left(\id+\sigma_k B^{-1}\right)^{-1}\left[y^k+\sigma_k L\left(x^k+\theta_k(x^k-x^{k-1})\right)\right].
\end{eqnarray*}
This together with \eqref{C-var-x-acc-equiv} gives rise for all $k \geq 0$ to the following numerical scheme 
\begin{eqnarray} x^{k+1} & = &
J_{(\tau_{k+1}/\lambda)A}\left[x^k+\frac{\tau_{k+1}}{\lambda}\left(-L^*y^{k+1}-Cx^k\right)\right] \label{bchh1}\\
y^{k+2} & = &  J_{\sigma_{k+1}B^{-1}}\left[y^{k+1}+\sigma_{k+1} L\left(x^{k+1}+\theta_{k+1}(x^{k+1}-x^{k})\right)\right]\label{bchh2},
\end{eqnarray}
which has been investigated by Bo\c t, Csetnek, Heinrich and Hendrich in \cite[Algorithm 5]{b-c-h2}. 
Not least, assuming that $C=0$ and $\lambda=1$, the variational case $A=\partial f$ and $B=\partial g$  leads for all $k \geq 0$ to the numerical scheme 
\begin{eqnarray} y^{k+1} & = &  \prox\nolimits_{\sigma_k g^*}\left[y^{k}+\sigma_{k} L\left(x^{k}+\theta_{k}(x^{k}-x^{k-1})\right)\right]\\ \label{ch-p-acc1}
x^{k+1} & = &\prox\nolimits_{\tau_{k+1}f}\left(x^k-\tau_{k+1}L^*y^{k+1}\right)\label{ch-p-acc2},
\end{eqnarray}
which has been considered by Chambolle and Pock in \cite[Algorithm 2]{ch-pck}.

We also notice that condition \eqref{choice-pd} guarantees the fulfillment of both  \eqref{mon1} and \eqref{mon2}, due to the fact that 
the sequence $(\tau_{k+1}\sigma_k)_{k\geq 0}$ is constant (see \eqref{tau} and \eqref{sigma}). 
\end{remark}

\begin{remark}\label{rem-admm-acc} Assume again that $C=0$ and consider the variational case as described in Problem \ref{pr-var}. From 
\eqref{c1} and \eqref{def-m_tild} we derive for all $k \geq 1$ the relation 
$$0\in\partial g^*(y^{k+1})+\tau_k L\left(L^*y^{k+1}+z^k-\tau_k^{-1}x^k\right)+M_2^k\left(y^{k+1}-y^k\right),$$
which in case $M_2^k\in {\cal S_+}({\cal G})$ is equivalent to
$$y^{k+1}=\argmin_{y\in {\cal G}}\left[g^*(y)+\frac{\tau_k}{2}\left\|L^*y+z^k-\tau_k^{-1}x^k\right\|^2+\frac{1}{2}\|y-y^k\|_{M_2^k}^2\right].$$
Algorithm \ref{alg-mon-acc} becomes in case $\lambda =1$
\begin{eqnarray*} y^{k+1} & = &
\argmin_{y\in {\cal G}}\left[g^*(y)+\frac{\tau_k}{2}\left\|L^*y+z^k-\tau_k^{-1}x^k\right\|^2+\frac{1}{2}\|y-y^k\|_{M_2^k}^2\right] \label{C-var-y-acc-var}\\
 z^{k+1} & = &\left(\theta_k-1\right)L^*y^{k+1} + \theta_k \argmin_{x\in {\cal H}}
\left[f^*(x)+\frac{\tau_{k+1}}{2}\left\|-L^*y^{k+1}-z+\tau_{k+1}^{-1}x^k\right\|^2\right] \label{C-var-z-acc-var}\\
x^{k+1} & = & x^k+\frac{\tau_{k+1}}{\theta_k}\left(-L^*y^{k+1}-z^{k+1}\right),\label{C-var-x-acc-var}
\end{eqnarray*}
which can be regarded as an accelerated version of the algorithm \eqref{h-var-x}-\eqref{h-var-y} 
in Remark \ref{bbc-sh-teb}. 
\end{remark}

We present the main theorem of this section. 

\begin{theorem}\label{th-acc} Consider the setting  of Problem \ref{pr-acc} and let $(x,v)$ be a primal-dual solution to the primal-dual system of monotone inclusions
\eqref{p-mon}-\eqref{d-mon}. Let $(x^k,z^k,y^k)_{k\geq 0}$ be the sequence generated by Algorithm \ref{alg-mon-acc} and assume that 
the relations \eqref{tau1}-\eqref{mon2} are fulfilled. 
Then we have for all $n\geq 2$
\begin{eqnarray*}
& & \frac{\lambda\|x^{n}-x\|^2}{\tau_{n+1}^2}+\frac{1-\sigma_0\tau_1\|L\|^2}{\sigma_0\tau_1}\|y^n-v\|^2 \leq\\
& & \frac{\lambda\|x^1-x\|^2}{\tau_2^2}+\frac{\|y^{1}-v\|_{\tau_1 LL^*+M_2^1}^2}{\tau_2} +\frac{\|x^1-x^0\|^2}{\tau_1^2} + 
\frac{2}{\tau_1}\langle L(x^1-x^0),y^1-v\rangle.\end{eqnarray*}
Moreover, $\lim\limits_{n\rightarrow+\infty}n\tau_n=\frac{\lambda}{\gamma}$, hence one obtains for $(x^n)_{n \geq 0}$ an 
order of convergence of ${\cal {O}}(\frac{1}{n})$. 
\end{theorem}

\begin{proof} Let $k\geq 1$ be fixed. From \eqref{C-var-y-acc-equiv}, the relation $Lx\in B^{-1}v$ (see \eqref{prim-dual}) 
and the monotonicity of $B^{-1}$ we obtain
$$\left\langle y^{k+1}-v, \widetilde{M_2}^k(y^k-y^{k+1})+L\left[x^k+\theta_{k-1}(x^k-x^{k-1})\right]-Lx\right\rangle\geq 0$$
or, equivalently,
\begin{equation}\label{y1} \frac{1}{2}\|y^k-v\|_{\widetilde{M_2}^k}^2-\frac{1}{2}\|y^{k+1}-v\|_{\widetilde{M_2}^k}^2
-\frac{1}{2}\|y^k-y^{k+1}\|_{\widetilde{M_2}^k}^2\geq 
\left\langle y^{k+1}-v,Lx-L\left[x^k+\theta_{k-1}(x^k-x^{k-1})\right]\right\rangle.
\end{equation}
Further, from \eqref{C-var-x-acc-equiv-def-res}, the relation $-L^* v\in (A+C) x$ (see \eqref{prim-dual}) and the 
$\gamma$-strong monotonicity of $A+C$ we obtain 
$$\left\langle x^{k+1}-x, \frac{\lambda}{\tau_{k+1}}\left(x^k-x^{k+1}\right)-L^*y^{k+1}+Cx^{k+1}-Cx^k+L^*v\right\rangle
\geq \gamma\|x^{k+1}-x\|^2$$
or, equivalently,
\begin{eqnarray}\label{ineqx} &\frac{\lambda}{2\tau_{k+1}}\|x^k-x\|^2-\frac{\lambda}{2\tau_{k+1}}\|x^{k+1}-x\|^2
-\frac{\lambda}{2\tau_{k+1}}\|x^k- x^{k+1}\|^2\geq& \gamma\|x^{k+1}-x\|^2\nonumber\\&&+\langle x^ {k+1}-x,Cx^k-Cx^{k+1}\rangle\nonumber\\
&&+\langle y^{k+1}-v,Lx^{k+1}- x\rangle.
\end{eqnarray}
Since $C$ is $\mu$-Lipschitz continuous, we have that 
$$\langle x^ {k+1}-x,Cx^k-Cx^{k+1}\rangle\geq -\frac{\mu\tau_{k+1}}{2}\|x^{k+1}-x\|^2-\frac{\mu}{2\tau_{k+1}}\|x^{k+1}-x^k\|^2,$$
which combined with \eqref{ineqx} implies
\begin{eqnarray}\label{ineqx-2} \frac{\lambda}{2\tau_{k+1}}\|x^k-x\|^2
&\geq& \left(\frac{\lambda}{2\tau_{k+1}}+\gamma-\frac{\mu\tau_{k+1}}{2}\right)\|x^{k+1}-x\|^2
+\frac{\lambda-\mu}{2\tau_{k+1}}\|x^{k+1}- x^k\|^2\nonumber\\
&&+\langle y^{k+1}-v,Lx^{k+1}-Lx\rangle.
\end{eqnarray}
By adding the inequalities \eqref{y1} and \eqref{ineqx-2}, we obtain 
\begin{eqnarray}\label{xy-comb}\frac{1}{2}\|y^k-v\|_{\widetilde{M_2}^k}^2+ \frac{\lambda}{2\tau_{k+1}}\|x^k-x\|^2&\geq&
\frac{1}{2}\|y^{k+1}-v\|_{\widetilde{M_2}^k}^2+\left(\frac{\lambda}{2\tau_{k+1}}+\gamma-\frac{\mu\tau_{k+1}}{2}\right)\|x^{k+1}-x\|^2\nonumber\\
&&+\frac{1}{2}\|y^k-y^{k+1}\|_{\widetilde{M_2}^k}^2+\frac{\lambda-\mu}{2\tau_{k+1}}\|x^{k+1}- x^k\|^2\nonumber\\
&&+\left\langle y^{k+1}-\ol v,L \left[x^{k+1}-x^k-\theta_{k-1}(x^k-x^{k-1})\right]\right\rangle.
\end{eqnarray}
Further, we have 
\begin{align*}
\left\langle L\left[x^{k+1}-x^k-\theta_{k-1}(x^k-x^{k-1})\right],y^{k+1}-v\right\rangle =  \ &
\langle L(x^{k+1}-x^{k}),y^{k+1}-v\rangle\\
&-\theta_{k-1}\langle L(x^{k}-x^{k-1}),y^{k}-v\rangle \\
& + \theta_{k-1}\langle L(x^{k}-x^{k-1}),y^{k}-y^{k+1}\rangle \\
 \geq \ & \langle L(x^{k+1}-x^{k}),y^{k+1}-v\rangle\\
&-\theta_{k-1}\langle L(x^{k}-x^{k-1}),y^{k}- v\rangle\\
&  -\frac{\theta_{k-1}^2\|L\|^2\sigma_{k}}{2}\|x^{k-1}-x^k\|^2-\frac{\|y^k-y^{k+1}\|^2}{2\sigma_{k}}.
\end{align*}

By combining this inequality with \eqref{xy-comb}  we obtain (after dividing by $\tau_{k+1}$)
\begin{eqnarray}\frac{\|y^k-v\|_{\widetilde{M_2}^k}^2}{2\tau_{k+1}}+ \frac{\lambda}{2\tau^2_{k+1}}\|x^k-x\|^2&\geq&
\frac{\|y^{k+1}-v\|_{\widetilde{M_2}^k}^2}{2\tau_{k+1}}+\left(\frac{\lambda}{2\tau^2_{k+1}}+\frac{\gamma}{\tau_{k+1}}-\frac{\mu}{2}\right)\|x^{k+1}-x\|^2\nonumber\\
&&+\frac{\|y^k-y^{k+1}\|_{\widetilde{M_2}^k}^2}{2\tau_{k+1}}-\frac{\|y^k-y^{k+1}\|^2}{2\tau_{k+1}\sigma_{k}}\label{xy-comb2}\\
&&+\frac{\lambda-\mu}{2\tau^2_{k+1}}\|x^{k+1}- x^k\|^2-\frac{\theta_{k-1}^2\|L\|^2\sigma_{k}}{2\tau_{k+1}}\|x^k-x^{k-1}\|^2\nonumber\\
&&+\frac{1}{\tau_{k+1}}\langle L(x^{k+1}-x^{k}),y^{k+1}-v\rangle\nonumber\\
&&-\frac{\theta_{k-1}}{\tau_{k+1}}\langle L(x^{k}-x^{k-1}),y^{k}-v\rangle\nonumber.
\end{eqnarray}
From \eqref{mon1} and \eqref{def-m_tild} we have that the term in \eqref{xy-comb2} is nonnegative. Further, 
noticing that (see \eqref{theta}, \eqref{tau}, \eqref{sigma} and \eqref{sigma0-tau1})
$$\frac{\theta_{k-1}}{\tau_{k+1}}=\frac{1}{\tau_{k}}$$
$$\frac{\lambda}{2\tau_{k+1}^2}+\frac{\gamma}{\tau_{k+1}}-\frac{\mu}{2}=\frac{\lambda}{2\tau_{k+2}^2},$$
$$\tau_{k+1}\sigma_{k}=\tau_k\sigma_{k-1}=...=\tau_1\sigma_{0}$$ and
$$\frac{\|L\|^2\sigma_{k}\theta_{k-1}^2}{\tau_{k+1}}=
\frac{\tau_{k+1}\|L\|^2\sigma_{k}}{\tau_{k}^2}=
\frac{\tau_1\|L\|^2\sigma_{0}}{\tau_k^2}\leq \frac{1}{\tau_k^2},$$
we obtain (see also \eqref{lambda} and \eqref{mon2})
\begin{eqnarray*}\frac{\|y^k-v\|_{\widetilde{M_2}^k}^2}{2\tau_{k+1}}+ \frac{\lambda}{2\tau^2_{k+1}}\|x^k-x\|^2 \!\! &\geq&
\frac{\|y^{k+1}-v\|_{\widetilde{M_2}^{k+1}}^2}{2\tau_{k+2}}+\frac{\lambda}{2\tau_{k+2}^2}\|x^{k+1}-x\|^2\nonumber\\
&&+\frac{1}{2\tau^2_{k+1}}\|x^{k+1}- x^k\|^2-\frac{1}{2\tau^2_{k}}\|x^k-x^{k-1}\|^2\nonumber\\
&&+\frac{1}{\tau_{k+1}}\langle L(x^{k+1}-x^{k}),y^{k+1}-v\rangle \!-\!\frac{1}{\tau_{k}}\langle L(x^{k}-x^{k-1}),y^{k}-v\rangle\nonumber.
\end{eqnarray*}
Let $n$ be a natural number such that $n\geq 2$. Summing up the above inequality from $k=1$ to $n-1$, it follows
\begin{eqnarray*}\frac{\|y^1- v\|_{\widetilde{M_2}^1}^2}{2\tau_{2}}+ \frac{\lambda}{2\tau^2_{2}}\|x^1- x\|^2&\geq&
\frac{\|y^{n}- v\|_{\widetilde{M_2}^{n}}^2}{2\tau_{n+1}}+\frac{\lambda}{2\tau_{n+1}^2}\|x^{n}- x\|^2\nonumber\\
&&+\frac{1}{2\tau^2_{n}}\|x^{n}- x^{n-1}\|^2-\frac{1}{2\tau^2_{1}}\|x^1-x^{0}\|^2\nonumber\\
&&+\frac{1}{\tau_{n}}\langle L(x^{n}-x^{n-1}),y^{n}- v\rangle-\frac{1}{\tau_{1}}\langle L(x^{1}-x^{0}),y^{1}- v\rangle\nonumber.
\end{eqnarray*}
The inequality in the statement of the theorem follows by combining this relation with (see \eqref{mon1})
$$\frac{\|y^{n}- v\|_{\widetilde{M_2}^{n}}^2}{2\tau_{n+1}}\geq \frac{\|y^{n}- v\|^2}{2\sigma_n\tau_{n+1}},$$
$$\frac{1}{2\tau^2_{n}}\|x^{n}- x^{n-1}\|^2+\frac{1}{\tau_{n}}\langle L(x^{n}-x^{n-1}),y^{n}- v\rangle\geq 
-\frac{\|L\|^2}{2}\|y^n- v\|^2 \ \mbox{and} \ \sigma_n\tau_{n+1}=\sigma_0\tau_1.$$ 

Finally, we notice that for any $n\geq 0$ (see \eqref{theta} and \eqref{tau})
\begin{equation}\label{tau-n}\tau_{n+2}=\frac{\tau_{n+1}}{\sqrt{1+\frac{\tau_{n+1}}{\lambda}(2\gamma-\mu\tau_{n+1})}}.\end{equation}
From here it follows that $\tau_{n+1}<\tau_n$ for all $n\geq 1$ and 
$\lim\limits_{n\rightarrow+\infty}n\tau_n=\lambda/\gamma$ (see \cite[page 261]{b-c-h2}). The proof is complete. 
\end{proof}

\begin{remark}\label{param-choice} In Remark \ref{param-choice-pf-acc} we provided an example of a 
family of linear, continuous and self-adjoint operators  $(M_2^k)_{k \geq 0}$ for which the relations \eqref{mon1} and \eqref{mon2} are fulfilled. In the following we will furnish more examples
in this sense. 

To begin we notice that simple algebraic manipulations easily lead to the conclusion that if 
\begin{equation}\label{tau1-stronger} \mu\tau_1< \gamma,
\end{equation}
then $(\theta_k)_{k\geq 0}$ is monotonically increasing. In the examples below we replace \eqref{tau1} with the stronger assumption \eqref{tau1-stronger}.

(i) For all $k\geq 0$, take $$M_2^k:=\sigma_k^{-1}\id.$$
Then \eqref{mon1} trivially holds, while \eqref{mon2}, which can be equivalently written as 
$$\frac{1}{\theta_{k-1}}LL^*+\frac{1}{\tau_{k+1}}M_2^k\succcurlyeq 
\frac{1}{\theta_{k}} LL^*+\frac{1}{\tau_{k+2}}M_2^{k+1},$$
follows from the fact that $(\tau_{k+1}\sigma_k)_{k\geq 0}$ is constant (see \eqref{tau} and\eqref{sigma}) 
and $(\theta_k)_{k\geq 0}$ is monotonically increasing.

(ii) For all $k\geq 0$, take $$M_2^k:=0.$$
Relation \eqref{mon2} holds since $(\theta_k)_{k\geq 0}$ is monotonically increasing. 
Condition \eqref{mon1} becomes in this setting 
\begin{equation}\label{mon1-m0} \sigma_k\tau_k LL^*\succcurlyeq  \id \ \forall k \geq 0.
\end{equation}
Since $\tau_k>\tau_{k+1}$ for all $k\geq 1$ and $(\tau_{k+1}\sigma_k)_{k\geq 0}$ is constant, \eqref{mon1-m0} holds, if 
\begin{equation}\label{cond} LL^*\in{\cal P}_{\frac{1}{\sigma_0\tau_1}}({\cal G}). 
\end{equation}
In this case one has to take in \eqref{sigma0-tau1} 
$$\sigma_0\tau_1\|L\|^2=1.$$
In view of the above theorem, the iterative scheme obtained in this particular instance (see Remark \ref{rem-admm-acc}) 
can be regarded as an accelerated version of the classical ADMM algorithm (see Remark \ref{bbc-sh-teb} and Remark \ref{rem-th1}(ii)). 

(iii) For all $k\geq 0$, take $$M_2^k:=\tau_{k}\id.$$
Relation \eqref{mon2} holds since $(\theta_k)_{k\geq 0}$ is monotonically increasing. On the other hand, condition \eqref{mon1} is equivalent to 
\begin{equation}\label{cond2} \sigma_k\tau_k(LL^*+\id)\succcurlyeq  \id.
\end{equation}
Since $\tau_k>\tau_{k+1}$ for all $k\geq 1$ and $(\tau_{k+1}\sigma_k)_{k\geq 0}$ is constant, \eqref{cond2} holds, if 
\begin{equation}\label{cond3} \sigma_0\tau_1 LL^*\succcurlyeq  (1-\sigma_0\tau_1)\id.\end{equation}
In case $\sigma_0\tau_1\geq 1$ (which is allowed according to \eqref{sigma0-tau1} if $\|L\|^2\leq 1$) this is obviously fulfilled. 
Otherwise, in order to guarantee \eqref{cond3}, we have to impose that 
\begin{equation}\label{cond4}  LL^*\in{\cal P}_{\frac{1-\sigma_0\tau_1}{\sigma_0\tau_1}}({\cal G}).\end{equation}

\end{remark}

\end{document}